\documentclass{birkjour}
 \newtheorem{theorem}{Theorem}[section]
 
 \newtheorem{lemma}[theorem]{Lemma}
 
 \theoremstyle{definition}
 
 \theoremstyle{remark}
 
 \newtheorem{example}{Example}
 \numberwithin{equation}{section}
\usepackage{hyperref}
\usepackage{mathrsfs}
\usepackage{amsmath,geometry,amssymb}
\usepackage{graphicx}

\begin{document}

%-------------------------------------------------------------------------
% editorial commands: to be inserted by the editorial office
%
%\firstpage{1} \volume{228} \Copyrightyear{2004} \DOI{003-0001}
%
%
%\seriesextra{Just an add-on}
%\seriesextraline{This is the Concrete Title of this Book\br H.E. R and S.T.C. W, Eds.}
%
% for journals:
%
%\firstpage{1}
%\issuenumber{1}
%\Volumeandyear{1 (2004)}
%\Copyrightyear{2004}
%\DOI{003-xxxx-y}
%\Signet
%\commby{inhouse}
%\submitted{March 14, 2003}
%\received{March 16, 2000}
%\revised{June 1, 2000}
%\accepted{July 22, 2000}
%
%
%
%---------------------------------------------------------------------------
%Insert here the title, affiliations and abstract:
%

\title[]
 {Conformal Ricci solitons on generalized $(\kappa,\mu)$-space forms}

%----------Author 1

\author[Lone]{Mehraj Ahmad Lone}

\address{%
	Department of Mathematics,\\National Institute of Technology Srinagar, \\
	190006, Kashmir, India.}

\email{mehrajlone@nitsri.net}

\author[Wani]{Towseef Ali Wani}

\address{%
	Department of Mathematics,\\National Institute of Technology Srinagar, \\
	190006, Kashmir, India.}

\email{towseef\_02phd19@nitsri.net}

%\thanks{This work was completed with the support of our
%\TeX-pert.}
%----------Author 2

%----------classification, keywords, date


%\date{March 30, 2017}
%----------additions
%\dedicatory{To my boss}
%%% ----------------------------------------------------------------------

\begin{abstract}
In this paper, we study conformal Ricci solitons and conformal gradient Ricci solitons on generalized $(\kappa,\mu)$-space forms. The conditions for the solitons to be shrinking, steady and expanding are derived in terms of conformal pressure $p$. We show under what conditions a Ricci semi-symmetric generalized $(\kappa,\mu)$- space form equiped with a conformal Ricci soliton forms an Einstein manifold.
\end{abstract}

%%% ----------------------------------------------------------------------
\maketitle
%%% ----------------------------------------------------------------------

\section{\protect \bigskip Introduction} 	
The concept of Ricci flow was introduced by R. Hamilton \cite{Hamilton} in 1982. The Ricci flow is an evolution equation for metric on a Riemannian manifold given by
\begin{align*}
	\frac{\partial g}{\partial t}=-2 S
\end{align*}
where $g$ is the Riemannian metric and $S$ denotes the Ricci tensor.\\
A self-similar solution of the Ricci flow \cite{Hamilton,Topping}, which moves only by one parameter family of diffeomorphism and scaling is called a Ricci soliton \cite{Hamilton_S}. The Ricci soliton is given by
\begin{align*}
	\mathcal{L}_Vg+2S=2\lambda g
\end{align*}
where $\mathcal{L}_V$ is the Lie derivative, $S$ is the Ricci tensor, $g$ is the Riemannian metric, $V$ is the vector field and $\lambda$ is a scalar. The Ricci soliton is denoted by $(g,V,\lambda)$ and is said to be shrinking, steady and expanding according to whether $\lambda$ is positive, zero and negative respectively.\\
The concept of conformal Ricci flow was introduced by Fischer \cite{Fischer} as a variation of classical Ricci flow equation that modifies the volume constraint to a scalar curvature constraint. The confomal Ricci flow on a smooth, closed, connected, oriented $n$-manifold is defined by the equation \cite{Fischer}
\begin{align*}
	&\frac{\partial g}{\partial t}+2(S+\frac{g}{n})=-pg\\&
	\text{and}\hspace{1cm}r=-1
\end{align*}
where $p$ is a non-dynamical scalar field which is time dependent, $r$ is the scalar curvature of the manifold, and $n$ is the dimension of $M$.\\
In 2015, Basu et. al. \cite{Basu} intoduced the notion of conformal Ricci soliton equation on Kenmotsu manifold $M^{2n+1}$ as
\begin{align}\label{1.1}
	\mathcal{L}_Vg+2S=[2\lambda-(p+\frac{2}{2n+1})]g
	\end{align}
	where $\lambda$ is constant. The equation is a generalization of the Ricci soliton and satisfies the conformal Ricci flow equation.The conformal Ricci flow equations are analogous
	to the Navier-Stokes equations of fluid mechanics and because of this analogy, the time-dependent scalar field $p$ is called a conformal pressure and, as for the real physical pressure in fluid mechanics that serves to maintain the incompressibility of the fluid, the
	conformal pressure serves as a Lagrange multiplier to conformally deform the metric
	flow so as to maintain the scalar curvature constraint.\\
	A conformal Ricci soliton is called a conformal gradient Ricci soliton if the potential vector field $V$ is gradient of some smooth function $f$ i.e. $V=grad(f)=\nabla f$ and satisfies
	\begin{align}
		\nabla \nabla f+S=[2\lambda-(p+\frac{2}{2n+1})]g
	\end{align}
	where $\nabla$ is Riemannian connection on the Riemannian manifold.\\
	Conformal Ricci solitons were studied by Ganguly et.al. within the framework of almost co-Kahler manifolds \cite{13} and $(LCS)_n$-manifolds \cite{14}. The authors generalized conformal Ricci solitons and obtained interesting rgesults in \cite{11}. Dey et.al. \cite{dey} studied conformal Ricci solitons on almost Kenmotsu manifolds.  Siddiqui \cite{Danish} studied conformal Ricci solitons of Lagrangian submanifolds in Kahler manifolds. Ganguly et. al. \cite{akram} investigated conformal Ricci solitons and quasi-Yamable soliton on generalized Sasakian space form. Motivated by these studies, we investigate conformal Ricci soliton and conformal gradient Ricci soliton in generalized $(\kappa,\mu)$- space forms.
\section{Preliminaries}
Let $(M,g)$ be a Riemannian manifold of dimension $(2n+1)$. $(M,g)$ is called almost contact manifold \cite{Blair}  if we can define an endomorpism $\phi$ on its tangent bundle $TM$, a vector field $\xi$ and a 1-form $\eta$ satisfying
\begin{eqnarray}
\phi^2 = -Id + \eta \otimes \xi, \quad  \phi(\xi)=0,\quad \eta(\phi)=0,
\end{eqnarray}
\begin{eqnarray}
g(X,\xi)= \eta(X),
\end{eqnarray}
\begin{eqnarray}\label{2.3}
\eta(\xi)=1,
\end{eqnarray}
for any vector fields $X, Y$ on $M$.
It is called a contact metric manifold if 
\begin{eqnarray*}
g(\phi X, \phi Y) = g(X,Y)-\eta(X)\eta(Y).
\end{eqnarray*}
Making use of above equations, it is easy to prove that for an almost contact metric manifold $(M,g)$,
\begin{eqnarray}
g(\phi E,F)= -g(E,\phi  F).
\end{eqnarray}
An almost contact metric manifold is said to be a contact manifold if its second fundamental 2-form $\Phi$, defined by $\Phi(X,Y)= g(X,\phi Y)$, satisfies
\begin{eqnarray*}
	d\eta= \Phi.
\end{eqnarray*}
On a contact metric manifold $M(\phi,\xi,\eta,g)$ the tensor $h$ defined by $\mathcal{L}_{\xi}\phi$ is symmetric and satisfies the following relations \cite{Blair}
\begin{align*}
\nabla_X\xi=-\phi X-\phi hX, h\xi=0,h\phi=-\phi h, tr(h)=0,\eta\circ h=0
\end{align*}
A contact manifold is called a $(\kappa,\mu)$-metric manifold if the characteristic vector field $\xi$ belongs to the $(\kappa,\mu)$-distribution i.e.\cite{BOeckx}
\begin{align*}
	R(X,Y)\xi=\kappa\{\eta(Y)X-\eta(X)Y\}-\mu\{\eta(Y)hX-\eta(X)hY\}
\end{align*}
where $X$ and $Y$ are vector fields on $M$ and $2h=\mathcal{L}_{\xi}\phi$. If $\kappa,\mu $ are some smooth functions then the manifold is called generalized $(\kappa,\mu)$-space. A $(\kappa,\mu)$-space of dimension greater than three with constant $\phi$-sectional curvature is called $(\kappa,\mu)$-space form and its curvature tensor is given by \cite{Kouf}
\begin{align*}
R(X,Y)Z=&\frac{c+3}{4}R_1(X,Y)Z+\frac{c-1}{4}R_2(X,Y)Z+(\frac{c+3}{4}-\kappa)R_3(X,Y)Z\\&+R_4(X,Y)Z+\frac{1}{2}R_5(X,Y)Z+(1-\mu)R_6(X,Y)Z
\end{align*}
where $R_1,R_2, R_3, R_4, R_5, R_6$ are defined as follows
\begin{align*}
	&R_1(X,Y)Z=g(Y,Z)X-g(X,Z)Y,\\&
	R_2(X,Y)Z=g(X,\phi Z)\phi Y-g(Y,\phi Z)\phi X+2g(X,\phi Y)\phi Z,\\&
	R_3(X,Y)Z=\eta(X)\eta(Z)Y-\eta(Y)\eta(Z)X+g(X,Z)\eta(Y)\xi-g(Y,Z)\eta(X)\xi,\\&
	R_4(X,Y)Z= g(Y,Z)hX-g(X,Z)hY+g(hY,Z)X-g(hX,Z)Y,\\&
	R_5(X,Y)Z=g(hY,Z)hX-g(hX,Z)hY+g(\phi hX,Z)\phi hY-g(\phi hY,Z)\phi hX,\\&
	R_6(X,Y)Z=\eta(X)\eta(Z)hY-\eta(Y)\eta(Z)hX+g(hX,Z)\eta(Y)\xi-g(hY,Z)\eta X\xi,
\end{align*}
for any vector fields $X, Y, Z$ on $M$.
As a generalization of $(\kappa,\mu)$-space form, Carriazo et. al. \cite{Tripathi} introduced the notion of generalized $(\kappa,\mu)$-space form and provided interesting examples of such spaces. An almost contact metric manifold is called a generalized $(\kappa,\mu)$-space form if there exist smooth functions $f_1, f_2, f_3, f_4, f_5, f_6$ such that
\begin{align*}
	R(X,Y)Z=f_1R_1(X,Y)Z&+f_2R_2(X,Y)Z+f_3R_3(X,Y)Z+f_4R_4(X,Y)Z\\&+f_5R_5(X,Y)Z+f_6R_6(X,Y)Z
\end{align*}
where $R_1, R_2,R_3,R_4,R_5,R_6$ are defined as above. 
\begin{example}\cite{Kouf_TSI}
	Consider the 3-dimensional manifold $M=\{(x_1,x_2,x_3)\in \mathbb{R}^3|x_3\neq0\}$, where $(x_1,x_2,x_3)$ are the standard coordinates in $\mathbb{R}^3$. The vector fields,
	$$e_1=\frac{\partial}{\partial x_1}, \hspace{.25cm}e_2=-2x_2x_3\frac{\partial}{\partial x_1}+2\frac{x_1}{{x_3}^3}\frac{\partial}{\partial x_2}-\frac{1}{x_3^2}\frac{\partial}{\partial x_3}, e_3=\frac{1}{x_3}\frac{\partial}{\partial x_2},$$
	are linearly independent at each point of $M$. Let $g$ be the Riemannian metric defined by $g(e_i,e_j)=\delta_{ij}, i,j=1,2,3.$ Let $\nabla$ be the Riemannian connection and $R$ the curvature tensor of $g$. We easily get
	$$[e_1,e_2]=\frac{2}{{x_3}^2}e_3,\hspace{.25cm}[e_2,e_3]=2e_1+\frac{1}{{x_3}^3}e_3,\hspace{.25cm} [e_3,e_1]=0.$$
	Let $\eta$ be the 1-form defined by $\eta(z)=g(z,e_1)$ for any $z \in \mathcal{X}(M).$ Because $\eta \wedge d\eta \neq 0,$ everywhere on $M$, $\eta$ is a contact form. Let $\phi$ be the $(1,1)$-tensor field, defined by $\phi e_1=0, \hspace{.15cm}\phi e_2=e_3, \hspace{.15cm}\phi e_3=-e_2$. Using the linearity of $\phi, d\eta,$ and $g$, we have $$\eta(e_1)=1, \hspace{.15cm}\phi^2z=z-\eta(z)e_1,\hspace{.15cm} d\eta(z,w)=g(z,\phi w)$$
	and
	$$g(\phi z,\phi w)=g(z,w)-\eta(z)\eta(w)$$
	 for any $z, w \in \mathcal{X}(M).$
	 Hence $(\phi,e_1,\eta,g)$ defines a contact metric structure $M$. So $M$ with this structure is a contact metric manifold. \\
	 Putting $\xi=e_1$, $x=e_2$, $\phi x=e_3$ and using the well known formula
	 \begin{eqnarray*}
	 	2g(\nabla_yz,w)&=&yg(z,w)+zg(w,y)-wg(y,z)-g\big(y,[z,w]\big)\\&&-g\big(z,[y,w]\big)+g\big(w,[y,z]\big),
	 \end{eqnarray*}
	 
	 We calculate
	 $$\nabla_x\xi=-(1+\frac{1}{{x_3}^2})\phi x,\hspace{.25 cm}\nabla_{\phi x}\xi=(1-\frac{1}{{x_3}^2})x,$$
	 $$\nabla_\xi x=(-1+\frac{1}{{x_3}^2})\phi x,\hspace{.25 cm} \nabla_\xi\phi x=(1-\frac{1}{{x_3}^2}) x,$$
	 $$\nabla_xx=0,\hspace{.25 cm}\nabla_x\phi x=(1+\frac{1}{{x_3}^2})\xi,$$
	 $$\nabla_{\phi x}x=(-1+\frac{1}{{x_3}^2})\xi-\frac{1}{{x_3}^3}\phi x,\hspace{.25 cm} \nabla_{\phi x}\phi x=\frac{1}{{x_3}^3}x.$$
	 Therefore for the tensor field $h$, we get $h\xi=0$, $hx=\lambda x$, and $\kappa=\frac{{x_3}^4-1}{{x_3}^4}$, we finally get
	 $$ R(x,\xi)\xi=\kappa\big(\eta(\xi)x-\eta(x)\xi\big)+\mu \big(\eta(\xi)hx-\eta(x)h\xi \big)$$
	 $$R(\phi x,\xi)\xi=\kappa\big(\eta(\xi)\phi x-\eta(\phi x)\xi\big)+\mu \big(\eta(\xi)h\phi x-\eta(\phi x)h\xi \big)$$
	  $$ R(x,\phi x)\xi=\kappa\big(\eta(\phi x)x-\eta(x)\phi x\big)+\mu \big(\eta(\phi x)hx-\eta(x)h\phi x \big).$$
	  These relations yield the following, by a straight forward calculations,
	  $$ R(z,w)\xi=\kappa\big(\eta(w)z-\eta(z)w\big)+\mu \big(\eta(w)hz-\eta(z)hw \big),$$
	  where $\kappa$ and $\mu$ are non-constant smooth functions. Hence $M$ is a generalized $(\kappa,\mu)$-contact metric manifold.
\end{example}
 Submanifolds of generalized $(\kappa,\mu)$-space forms were studied by Hui et. al. \cite{Hui}. Lee et. al. \cite{Vilcu} studied generalized Wintgen inequality for submanifolds in generalized $(\kappa,\mu)$-space forms.\\
For a generalized $(\kappa,\mu)$-space form, we have the following relations \cite{Tripathi}
\begin{align}\label{2.5}
	\nabla_X\xi=(f_3-f_1)\phi X+(f_6-f_4)\phi hX
	\end{align}
	\begin{align}\label{2.6}
		(\nabla_X\eta)Y=(f_3-f_1)g(\phi X,Y)+(f_6-f_4)g(\phi hX,Y)
	\end{align}
	\begin{align}\label{2.7}
		(\nabla_X\phi)Y=&\nonumber(f_1-f_3)[g(X,Y)\xi-\eta(Y)X]\\&+(f_4-f_6)[g(hX,Y)\xi-\eta(Y)hX]
	\end{align}
	\begin{align}\label{2.8}
			R(X,Y)\xi=&\nonumber(f_1-f_3)\{\eta(Y)X-\eta(X)Y\}\\&+(f_4-f_6)\{\eta(Y)hX-\eta(X)hY\}
	\end{align}
	\begin{align}\label{2.9}
		R(\xi,Y)Z=&\nonumber(f_1-f_3)[g(Y,Z)\xi-\eta(Z)Y]\\&+(f_4-f_6)[g(hY,Z)\xi-\eta(Z)hY]
	\end{align}
\begin{align}\label{2.10}
		S(X,Y)=&\nonumber(2nf_1+3f_2-f_3)g(X,Y)-\big(3f_2+(2n-1)f_3\big)\eta(X)\eta(Y)\\&+\big((2n-1)f_4-f_6\big)g(hX,Y)
\end{align}
\begin{align}\label{2.11}
		QX=&\nonumber(2nf_1+3f_2-f_3)X-\big(3f_2+(2n-1)f_3\big)\eta(X)\xi\\&+\big((2n-1)f_4-f_6\big)hX
\end{align}
\begin{align}\label{2.12}
		S(X,\xi)=2n(f_1-f_3)\eta(X),
\end{align}
\begin{align}\label{2.13}
	Q\xi=2n(f_1-f_3)\xi
\end{align}
for all  vector fields $X$ and $Y$ in $TM$ and where $S$ is the Ricci tensor and $Q$ is the Ricci operator related by $S(X,Y)=g\big(Q(X),Y\big)$.
\section{Main Results}
\begin{theorem}
Consider a generalized $(\kappa,\mu)$-space form $M$ admitting a conformal Ricci soliton $(g,V,\lambda)$. Then the soliton is 
\begin{enumerate}
\item shrinking if $p<[4n(f_3-f_1)-\frac{2}{2n+1}]$
\item steady if $p=[4n(f_3-f_1)-\frac{2}{2n+1}]$
\item expanding if $p>[4n(f_3-f_1)-\frac{2}{2n+1}]$
\end{enumerate}
\end{theorem}
\begin{proof}
Since $M$ admits a conformal Ricci soliton$(g,V,\lambda)$, from (\ref{1.1}) we have,
\begin{eqnarray*}
	(L_Vg)(X,Y)+2S(X,Y)+[2\lambda-(p+\frac{2}{2n+1})] g(X,Y)=0
\end{eqnarray*}
Using the property of Lie derivative, we get
\begin{eqnarray*}
	g(\nabla_XV,Y)+g(X,\nabla_YV)+2S(X,Y)+[2\lambda-(p+\frac{2}{2n+1})] g(X,Y)=0
\end{eqnarray*}
Substituting $X=Y=\xi$ and using (\ref{2.3}), we get
\begin{eqnarray*}
2g(\nabla_{\xi}V,\xi)+2S(\xi,\xi)+[2\lambda-(p+\frac{2}{2n+1})]=0
\end{eqnarray*}
Using the fact that $\nabla$ is a metric connection, we have
\begin{eqnarray*}
-2g(V,\nabla_{\xi}\xi)+ 2S(\xi,\xi)+[2\lambda-(p+\frac{2}{2n+1})]=0
\end{eqnarray*}
Since $\nabla_{\xi}\xi=0$, we get from the above equation
\begin{eqnarray*}
	2S(\xi,\xi)+[2\lambda-(p+\frac{2}{2n+1})]=0
\end{eqnarray*}
Now using (\ref{2.12}) and substituting the value of $S(\xi,\xi)$, we get
\begin{eqnarray*}
	4n(f_1-f_3)+[2\lambda-(p+\frac{2}{2n+1})]=0
\end{eqnarray*}
Rearranging the above equation, we get
\begin{eqnarray*}
\lambda=2n(f_1-f_3)+(\frac{p}{2}+\frac{1}{2n+1})
\end{eqnarray*}
Upon using three different conditions on $\lambda$ in the above equations, we get the desired expressions.
\end{proof}
\begin{theorem}
		Consider a $(2n+1)$-dimensional Ricci semi-symmetric generalized $(\kappa,\mu)$-space form $M$ admitting a conformal Ricci soliton $(g,V,\lambda)$. If $f_4=f_6$, then the manifold is Einstein and the potential vector field $V$ is a conformal vector field.
\end{theorem}
\begin{proof}
	
	Since the manifold $M$ is Ricci semi-symmetric, for any vector fields $X$ and $Y$ on $M$, we have
	\begin{eqnarray*}
		R(X,Y) . S=0,
	\end{eqnarray*}
	where $R$ is the curvature tensor and $S$ is the Ricci tensor of $M$.\\
	The above equation can be written as
	\begin{eqnarray*}
		S\big(R(X,Y)Z,U\big)+S\big(Z,R(X,Y)U\big)=0,
	\end{eqnarray*}
	where $X,Y,Z, U$ are vector fields on $M$.\\
	Replacing $U$ by $\xi$ in the above equation, we get
	\begin{eqnarray*}
		S\big(R(X,Y)Z,\xi\big)+S\big(Z,R(X,Y)\xi\big)=0
	\end{eqnarray*}
	Using the fact that $S(X,Y)=g(QX,Y)$ where $Q$ is the Ricci operator, we have
	\begin{eqnarray*}
		g\big(Q\xi,R(X,Y)Z\big)+S\big(Z,R(X,Y)\xi\big)=0
	\end{eqnarray*}
	Now using (\ref{2.8}) and (\ref{2.13}) in the avove equation, we get
	\begin{align*}
	g\big(&2n(f_1-f_3)\xi, R(X,Y)Z\big)+\\&S\big(Z,(f_1-f_3)[\eta(Y)X-\eta(X)Y]+(f_4-f_6)[\eta(Y)hX-\eta(X)hY]\big)=0
	\end{align*}
		Simplifying and make use of (\ref{2.3}) and the fact that $f_4=f_6$, we get
		\begin{align*}
		2n&(f_1-f_3)\eta\big(R(X,Y)Z\big)+(f_1-f_3)[\eta(Y)S(Z,X)-\eta(X)S(Z,Y)]=0
		\end{align*}
			Putting $X=\xi$, we get
			\begin{align*}
			2n(f_1-f_3)&\eta\big(R(\xi,Y)Z\big)+(f_1-f_3)[\eta(Y)S(Z,\xi)-S(Z,Y)]=0
			\end{align*}
				Making use of (\ref{2.9}), (\ref{2.12}) and then simplifying, we get
\begin{align}\label{3.2}
S(Y,Z)=2n(f_1-f_3)g(Y,Z),
\end{align}
Thus it is clear from the expression (\ref{3.2}) that $M$ is an Einstein manifold.\\
Since $M$ admits a conformal Ricci soliton$(g,V,\lambda)$, from (\ref{1.1}) we have,
\begin{eqnarray*}
	(L_Vg)(X,Y)+2S(X,Y)+[2\lambda-(p+\frac{2}{2n+1})] g(X,Y)=0
\end{eqnarray*}
Making use of (\ref{3.2}) in the above expression, we get
\begin{align*}
	(L_Vg)(X,Y)=[2\lambda-4n(f_1-f_3)-(p+\frac{2}{2n+1})] g(X,Y)
\end{align*}
Putting $\delta=[2\lambda-4n(f_1-f_3)-(p+\frac{2}{2n+1})]$, we can write
\begin{align}\label{3.3}
	\mathcal{L}_Vg=\delta g.
\end{align}
From (\ref{3.3}), we conclude that $V$ is a conformal vector field.
\end{proof}
\begin{theorem}
	Consider a $(2n+1)$-dimensional generalized $(\kappa,\mu)$-space form $M$ admitting a conformal Ricci soliton $(g,V,\lambda)$ whose potential vector field $V$ is pointwise collinear with the Reeb vector field $\xi$. Then V is constant multiple of $\xi$ and $M$ is an Einstein manifold of scalar curvature $r=2n(2n+1)(f_1-f_3)$.
\end{theorem}
\begin{proof}
Let's assume that $V= b\xi$ for some smooth function $b$. Then from (\ref{1.1}) we can write
\begin{align*}
bg(\nabla_X\xi,Y)+bg(X,\nabla_Y\xi)&+X(b)\eta(Y)+Y(b)\eta(X)+2S(X,Y)=\\&[2\lambda-(p+\frac{2}{2n+1})]g(X,Y)
\end{align*}
Using (\ref{2.5}) in the above equation, we get,
\begin{align}\label{3.4}
	X(b)\eta(Y)+Y(b)\eta(X)+2S(X,Y)=[2\lambda-(p+\frac{2}{2n+1})]g(X,Y)
\end{align}
Putting $Y=\xi$ and using (\ref{2.12}), we get,
\begin{align}\label{3.5}
	X(b)=[2\lambda-(p+\frac{2}{2n+1})-4n(f_1-f_3-\xi(b))]\eta(X)
\end{align}
Putting $X=\xi$, we get
\begin{align}\label{3.6}
	\xi(b)=[\lambda-(\frac{p}{2}+\frac{1}{2n+1})-2n(f_1-f_3)]
\end{align}
In view of (\ref{3.5}) and (\ref{3.6}), we can write
\begin{align}\label{3.7}
	db=[\lambda-(\frac{p}{2}+\frac{1}{2n+1})-2n(f_1-f_3)]\eta
\end{align}
Operating by $d$ on both sides of (\ref{3.7}) and uing the fact that $d^2=0$, we get
\begin{align}\label{3.8}
[\lambda-(\frac{p}{2}+\frac{1}{2n+1})-2n(f_1-f_3)]d\eta
\end{align}
Since $d\eta\neq 0$, we get from (\ref{3.8})
\begin{align}\label{3.9}
	\lambda=2n(f_1-f_3)+(\frac{p}{2}+\frac{1}{2n+1}).
\end{align}
Substituting (\ref{3.8})in (\ref{3.6}), we get $db=0$ which implies that $b$ is constant.Thus $V$ is constant multiple of $\xi$ which proves the first part of the theorem.\\
To  prove the second part of the theorem, we consider an orthonormal basis $\{e_i:1\leq i\leq 2n+1\}$ at each point of the manifold and put $X=Y=e_i$ in (\ref{3.4}) and summing over $1\leq i\leq(2n+1)$, we get
\begin{align}\label{3.10}
	\xi(b)+r=(2n+1)[\lambda-(\frac{p}{2}+\frac{1}{2n+1})].
\end{align}
Using the fact that $b$ is constant in (\ref{3.10})  we get,
\begin{align}\label{3.11}
	r=(2n+1)[\lambda-(\frac{p}{2}+\frac{1}{2n+1})].
\end{align}
Using (\ref{3.9}) in (\ref{3.11}), we get
\begin{align*}
	r=2n(2n+1)(g_1-g_3),
\end{align*}
which is the desired epression for the scalar curvature of the manifold.\\
Also, using the fact that $b$ is constant in(\ref{3.4}), we get
\begin{align*}
S(X,Y)=[\lambda-(\frac{p}{2}+\frac{1}{2n+1})]g(X,Y).
\end{align*}
Hence $M$ is an Einstein manifold.
\end{proof}
\begin{lemma}\label{l3.5}
Consider a $(2n+1)-$dimensional generalized $(\kappa,\mu)$-space form $M$ admitting a conformal gradient Ricci soliton $(g,\nabla f,\lambda)$. Then the curvature tensor $R$ satisfies\begin{small}
\begin{align*}
R(X,Y)\nabla f=& (2ndf_1+3df_2-df_3)(Y)X-(2ndf_1+3df_2-df_3)(X)Y\\&+\big(3f_2+(2n-1)f_3\big)(g_1-g_3)[2g(\phi X,Y)\xi+\eta(X)\phi X-\eta
(Y)\phi Y]\\&+\big(3f_2+(2n-1)f_3\big)(g_4-g_6)[2g(\phi hX,Y)\xi+\phi hX-\phi hY]\\&-\big(3df_2+(2n-1)df_3\big)(Y)\eta(X)\xi+\big(3df_2+(2n-1)df_3\big)(X)\eta(Y)\xi\\& +\big((2n-1)f_4-f_6\big)[(\nabla_Yh)X-(\nabla_Xh)Y]+\big((2n-1)df_4-df_6\big)(Y)hX\\&-\big((2n-1)df_4-df_6\big)(X)hY.
\end{align*}
\end{small}

\end{lemma}
\begin{proof}
Since $(g,\nabla f,\lambda)$ is a conformal gradient Ricci soliton on $N$, for any vector field $X$ on $N$, we can write
\begin{align*}
	\nabla_X\nabla f=[\lambda-(\frac{p}{2}+\frac{1}{2n+1})]X-QX,
\end{align*}
where $Q$ is the Ricci operator.
Differentiating covariantly with respect to arbitrary vector field $Y$, we get
\begin{align*}
\nabla_Y\nabla_X\nabla f=[\lambda-(\frac{p}{2}+\frac{1}{2n+1})]\nabla_YX-\nabla_YQX.
\end{align*}
Interchanging $X$ and $Y$ in the above equaton we get,
\begin{align*}
\nabla_X\nabla_Y\nabla f=[\lambda-(\frac{p}{2}+\frac{1}{2n+1})]\nabla_XY-\nabla_XQY.
\end{align*}
 Also, we can write
 \begin{align*}
 	\nabla_{[X,Y]}\nabla f=[\lambda-(\frac{p}{2}+\frac{1}{2n+1})](\nabla_XY-\nabla_YX)-Q(\nabla_XY-\nabla_YX).
 \end{align*} Using the above equations in the expression for curvature tensor we get
 \begin{align}\label{3.12}
 	R(X,Y)\nabla f=(\nabla_YQ)X-(\nabla_XQ)Y.
 \end{align}
 Now,
 \begin{align}\label{3.13}
 	(\nabla_YQ)X=\nabla_YQX-Q\nabla_YX.
 \end{align}
 Differentiating covariantly (\ref{2.11}) with respect to $Y$, we get
 \begin{align}\label{3.14}
 	\nabla_YQX=&\nonumber(2nf_1+3f_2-f_3)\nabla_YX+(2ndf_1+3df_2-df_3)(Y)X\\&\nonumber+\big((2n-1)f_4-f_6\big)\nabla_YhX+\big((2n-1)df_4-df_6\big)(Y)hX\\&\nonumber-\big(3f_2+(2n-1)f_3\big)[\nabla_Y\eta(X)\xi+\eta(X)\nabla_Y\xi]\\&-\big(3df_2+(2n-1)df_3\big)(Y)\eta(X)\xi.
 \end{align}
 Also, from (\ref{2.11}), we can write
 \begin{align}\label{3.15}
 	Q\nabla_YX=&\nonumber\big(2nf_1+3f_2-f_3\big)\nabla_XY+\big((2n-1)(f_4-f_6)h\nabla_YX\big)\\&-\big(3f_2+(2n-1f_3)\big)\eta(\nabla_YX)\xi.
 \end{align}
 Using (\ref{3.14}) and (\ref{3.15}) in (\ref{3.13}), and recalling (\ref{2.5}) and (\ref{2.6}) we get,
 \begin{align}\label{3.16}
 (\nabla_YQ)X=&\nonumber\big(2ndf_1+3df_2-df_3\big)(Y)X+\big((2n-1)f_4-f_6\big)(\nabla_Yh)X\\&\nonumber+\big((2n-1)df_4-df_6\big)(Y)hX\\&\nonumber-\big(3f_2+(2n-1)f_3\big)[(f_3-f_1)g(\phi X,Y)+(f_6-f_4)g(\phi hX,Y)]\xi\\&\nonumber-(3f_2+(2n-1)f_3)\eta(X)[(f_3-f_1)\phi X+(f_6-f_4)\phi hX]\\&-(3df_2+(2n-1)df_3)(Y)\eta(X)\xi.
 \end{align}
 Interchanging $X$ and $Y$ in (\ref{3.16}), we get
 \begin{align}\label{3.17}
 	(\nabla_XQ)Y=&\nonumber\big(2ndf_1+3df_2-df_3\big)(X)Y+\big((2n-1)f_4-f_6\big)(\nabla_Xh)Y\\&\nonumber+\big((2n-1)df_4-df_6\big)(X)hY\\&\nonumber-\big(3f_2+(2n-1)f_3\big)[(f_3-f_1)g(\phi Y,X)+(f_6-f_4)g(\phi hY,X)]\xi\\&\nonumber-(3f_2+(2n-1)f_3)\eta(Y)[(f_3-f_1)\phi Y+(f_6-f_4)\phi hY]\\&-(3df_2+(2n-1)df_3)(X)\eta(Y)\xi.
 \end{align}
 Now using (\ref{3.16}) and (\ref{3.17}) in (\ref{3.12}), we get the desired expression for $R(X,Y)\nabla f$
\end{proof}
\begin{theorem}
Consider a $(2n+1)$-dimensional generalized $(\kappa,\mu)$-space form $M$ admitting a conformal gradient Ricci soliton $(g,\nabla f,\lambda)$ such that $g(hY,\nabla f)=0$. Then the potential function f is constant if both $f_1$ and $f_3$ are constants.
\end{theorem}
\begin{proof}
	Since $M$ admits a conformal gradient Ricci solition, putting $X=\xi$ in (\ref{l3.5}), we get
	\begin{align}\label{3.18}
	R(\xi,Y)\nabla f=&\nonumber (2ndf_1+3df_2-df_3)(Y)\xi-(2ndf_1+3df_2-df_3)(\xi)Y\\&\nonumber+\big(3f_2+(2n-1)f_3\big)[(f_3-f_1)\eta
	(Y)\phi Y+(f_6-f_4)\phi hY]\\&\nonumber-\big(3df_2+(2n-1)df_3\big)(Y)\xi+\big(3df_2+(2n-1)df_3\big)(\xi)\eta(Y)\xi\\&\nonumber +\big((2n-1)f_4-f_6\big)[(\nabla_Yh)\xi-(\nabla_\xi h)Y]\\&-\big((2n-1)df_4-df_6\big)(\xi)hY.
	\end{align}
	Now taking the inner product of (\ref{3.18}) with structure vector field $\xi$, we get
	\begin{align}\label{3.19}
		g\big(R(\xi,Y)\nabla f,\xi\big)=&\nonumber 2n(df_1-df_3)(Y)-(2ndf_1+3df_2-df_3)(\xi)\eta(Y)\\&\nonumber+\big(3df_2+(2n-1)df_3\big)[(f_3-f_1)\eta
		(Y)g(\phi Y,\xi)\\&\nonumber+(f_6-f_4)g(\phi hY,\xi)]-\big(3df_2+(2n-1)df_3\big)(Y)\\&\nonumber+\big(3df_2+(2n-1)df_3\big)(\xi)\eta(Y)\\&\nonumber +\big((2n-1)f_4-f_6\big)g((\nabla_Yh)\xi-(\nabla_\xi h)Y,\xi)\\&-\big((2n-1)df_4-df_6\big)(X)g(hY,\xi).
	\end{align}
	Making use of the fact $h\xi=0$ and symmetry of $h$ in (\ref{3.19}), we get
	\begin{align}\label{3.20}
	g\big(R(\xi,Y)\nabla f,\xi\big)=2n(df_1-df_3)(Y)-2n(df_1-df_3)(\xi)\eta(Y).
	\end{align}
	Now, using the property of the curvature tensor, we have
	\begin{align}\label{3.21}
		g\big(R(\xi,Y)\nabla f,\xi\big)=-g\big(R(\xi,Y)\xi,\nabla f\big).
	\end{align}
	Making use of (\ref{2.9}) in (\ref{3.21}), we get
	\begin{align}\label{3.22}
		g\big(R(\xi,Y)\nabla f,\xi\big)=(f_3-f_1)[\eta(Y)(\xi f)-(Yf)]+(f_6-f_4)[-g(hY,\nabla f)].
	\end{align}
	Now using the fact that $g(hY,\nabla f)=0$, we get from the (\ref{3.22})
	\begin{align}\label{3.23}
			g\big(R(\xi,Y)\nabla f,\xi\big)=(f_3-f_1)[\eta(Y)(\xi f)-(Yf)].
	\end{align}
	From (\ref{3.20}) and (\ref{3.23}), we get
	\begin{align}\label{3.24}
		2n(df_1-df_3)(Y)-2n(df_1-df_3)(\xi)\eta(Y)=(f_3-f_1)[\eta(Y)(\xi f)-(Yf)].
	\end{align}
	If $f_1$ and $f_3$ are constants, then we get from (\ref{3.24})
	\begin{align*}
		\eta(Y)(\xi f)-(Yf)=0.
	\end{align*}
	The above equation can be rewritten as
	\begin{align*}
		g((\xi f)\xi,Y)=g(\nabla f,Y).
	\end{align*}
	Since $Y$ is an arbitrary vector field, we can write
	\begin{align}\label{3.25}
		\nabla f=(\xi f)\xi.
	\end{align}
	Differentiating (\ref{3.25}) covariantly along the vector field $X$, we get
	\begin{align*}
		\nabla_X\nabla f=\big(X(\xi f)\big)\xi+(\xi f)[(f_3-f_1)\phi X+(f_6-f_4)\phi hX].
	\end{align*}
	Substituting the value of $\nabla_X\nabla f$, we get
	\begin{align}\label{3.26}
		QX=[\lambda-(\frac{p}{2}+\frac{1}{2n+1})]X-\big(X(\xi f)\big)[(f_3-f_1)\phi X+(f_6-f_4)\phi hX].
	\end{align}
	Comparing the coefficients of $\phi X$ from (\ref{2.11}) and (\ref{3.26}), we get $(\xi f)=0$. Using this in (\ref{3.25}), we get $\nabla f$=0. Hence $f$ is constant.
\end{proof}
\begin{theorem}
Consider a $(2n+1)$-dimensional generalized $(\kappa,\mu)$-space form $M$ admitting a conformal gradient Ricci soliton $(g,\nabla f,\lambda)$.Then the soliton is 
\begin{enumerate}
	\item shrinking if $p<[2f_3-6f_2-2nf_1-\frac{2}{2n+1}]$
	\item steady if $p=[2f_3-6f_2-2nf_1-\frac{2}{2n+1}]$
	\item expanding if $p>[2f_3-6f_2-2nf_1-\frac{2}{2n+1}]$
\end{enumerate}
\end{theorem}
\begin{proof}
Comparing the coefficents of $X$ in (\ref{2.11}) and (\ref{3.26}), we get
\begin{align*}
	\lambda=[(\frac{p}{2}+\frac{1}{2n+1})+(2nf_1-3f_2-f_3)].
\end{align*}
Now we use the three different conditons on $\lambda$ in the above equation to get the desired expressions.
\end{proof}

\end{document}